\pdfoutput=1
\documentclass[11pt]{article}
\usepackage[left=1in,right=1in,top=1in,bottom=1in]{geometry}
\usepackage{times}
\usepackage{expl3}
\usepackage{cite}
\usepackage[table]{xcolor}
\usepackage{multirow}
\usepackage{stackengine} 
\usepackage{hhline}
\usepackage{lipsum}
\usepackage{titlesec}
\usepackage{wrapfig}
\usepackage{enumerate}
\usepackage{epsfig}
\usepackage{amsmath}
\usepackage{tabularx}
\usepackage{array}
\usepackage{booktabs}
\usepackage{enumitem}
\usepackage[ruled,vlined,linesnumbered]{algorithm2e}
\usepackage{bbm}
\usepackage{calc}
\usepackage{graphicx}
\usepackage{amsmath}
\usepackage[title]{appendix}
\usepackage{amssymb}
\usepackage{epstopdf}
\usepackage{boldline}
\usepackage{arydshln}
\usepackage{calligra}
\usepackage{bm}
\usepackage{url}
\usepackage{blindtext}
\usepackage{accents}
\usepackage{amsthm}

\newtheorem{theorem}{Theorem}

\newtheorem{corollary}{Corollary}
\newtheorem{lemma}{Lemma}
\newtheorem{problem}{Problem}[section] 
\newtheorem{example}{Example}
\newtheorem{remark}{Remark}
\usepackage{mathtools}
\usepackage{epstopdf}
\usepackage{balance}
\usepackage{thmtools}
\usepackage{thm-restate}
\usepackage{hyperref}
\usepackage{cleveref}
\usepackage[mathscr]{euscript}

\usepackage[ruled,vlined]{algorithm2e}
\include{pythonlisting}

\newcommand{\ostar}{\mathbin{\mathpalette\make@circled\star}}

\makeatletter
\newcommand{\removelatexerror}{\let\@latex@error\@gobble}
\makeatother
\makeatletter
\newcommand*{\rom}[1]{\expandafter\@slowromancap\romannumeral #1@}
\makeatother

\ExplSyntaxOn
\newcommand\latinabbrev[1]{
  \peek_meaning:NTF . {
    #1\@}%
  { \peek_catcode:NTF a {
      #1.\@ }%
    {#1.\@}}}
\ExplSyntaxOff

\titleclass{\subsubsubsection}{straight}[\subsubsection]

\begin{document}
\vspace{1cm}
\title{Existence and Enumeration of Polynomially Transformed Matrices under Spectral and Nilpotent Constraints}
\vspace{1.8cm}
\author{Shih-Yu~Chang
\thanks{Shih-Yu Chang is with the Department of Applied Data Science,
San Jose State University, San Jose, CA, U. S. A. (e-mail: {\tt
shihyu.chang@sjsu.edu}). 
           }}

\maketitle

\begin{abstract}
Matrix functions provide a unified framework that extends scalar functional concepts to linear operators, enabling broad applications across mathematics, science, and engineering. Classical formulations—through power series, spectral calculus, or Jordan canonical form—capture both diagonalizable and defective cases, allowing deeper insight into system dynamics, stability, and modal interactions. Building upon this foundation, this study investigates the transformation behavior of matrix functions under nilpotency and diagonalizability constraints, motivated by the Jordan decomposition of a matrix by the spectral part and the nilpotent part, where the spectral part  encodes eigenstructure and the nilpotent part characterizes deviation from diagonalizability. Two central problems are addressed: (1) identifying the conditions under which a polynomial transformed matrix becomes nilpotent, and (2) determining when a polynomial transformed matrix becomes diagonalizable. The approach proceeds through two primary scenarios—matrices with a single distinct eigenvalue and those with multiple distinct eigenvalues—each presenting different algebraic and combinatorial challenges. Depending on whether the polynomial transform function is sparse or dense, appropriate analytical tools are adopted: combinatorial root-counting techniques such as Descartes’ rule and Sturm’s theorem for sparse forms, and complex analytic or algebraic methods, including the argument principle, Rouché’s theorem, and square-free factorization, for dense cases. Collectively, these methods form a systematic framework for examining both the existence and enumeration of matrices satisfying prescribed structural properties of a polynomial transformed matrix. The results contribute to a deeper understanding of spectral transformations and offer new perspectives for applications in operator theory, dynamical systems, and computational linear algebra.
\end{abstract}
\begin{keywords}
Matrix functions, Jordan decomposition, nilpotent matrices, diagonalizable matrices, polynomial transformations.
\end{keywords}

\section{Introduction}\label{sec: Introduction}

Matrix functions extend scalar functions to matrix arguments and thereby form a unifying language for linear transformations across mathematics, science, and engineering. Classical constructions—via power series, the spectral (functional) calculus for diagonalizable matrices, or Jordan canonical form for defective cases—allow definitions of $\exp(A)$, $\log(A)$, $A^{1/2}$ and other maps that encapsulate system evolution, similarity transformations, and modal behavior \cite{Higham2008}. From a computational viewpoint, efficient and stable evaluation of matrix functions has motivated methods such as scaling-and-squaring with Padé approximants, Krylov subspace approaches for the action of a function on a vector, and rational approximations tailored to structure and sparsity \cite{Moler2003,AlMohy2011}. In applied mathematics, the matrix exponential furnishes closed-form solutions to linear ordinary differential systems and underpins exponential integrators; in quantum mechanics it represents time evolution operators; in control theory, transfer functions and Lyapunov analyses invoke matrix functions to assess stability and performance \cite{Golub2013}. Modern data and network sciences further exploit matrix functions to characterize diffusion, communicability, and centrality on graphs, while recent extensions to higher-order arrays prompt analogous functional calculus for tensors and multilinear operators \cite{chang2024operatorChar,chang2024generalizedCDJ}. Problems become more challenging when matrices are large, structured or uncertain: Numerical algorithm should balance the accuracy, complexities and preservation of qualitative properties. Overall, matrix functions can provide both rigorous analytical tools and practical computational mechanisms that bride down the principle and application in diverse scientific and engineering domains.

Given any matrix $\bm{X}$ with complex-valued entries, we can perform Jordan decomposition to $\bm{X}$. Then, we have $\bm{X} = \bm{S} + \bm{N}$, where $\bm{S}$ is the spectral (diagonalizable) part and $\bm{N}$ is the  nilpotent part. The spectral part $\bm{S}$ extracts the eigenstructure of $\bm{X}$—it determines invariant subspaces, eigenvalue multiplicities, and the asymptotic or steady-state behavior of linear transformations. On the other hand, the nilpotent part $\bm{N}$ encodes the deviation of $\bm{X}$ from diagonalizability; it governs transient dynamics, higher-order interactions between eigenvectors, and the algebraic multiplicity of eigenvalues~\cite{Putzer1966,chang2024operatorChar}.  

When a matrix function $f(\bm{X})$ is considered, both components contribute in fundamentally different ways. The spectral part transforms under $f$ according to the standard functional calculus on eigenvalues, whereas the nilpotent part interacts with the derivatives of $f$ evaluated at those eigenvalues. This decomposition therefore provides a natural analytical framework to explore two distinct but related problems:  
\begin{enumerate}[label=(\arabic*)]
    \item Under what conditions the transformed matrix $f(\bm{X})$ becomes nilpotent.
    \item When $f(\bm{X})$ becomes diagonalizable.
\end{enumerate}
Both problems assume that the function $f$ is provided. Beyond their theoretical significance, these questions motivate two types of investigations—existence and enumeration. Specifically, we seek to determine whether such matrices $\bm{X}$ exist for a given functional transformation and, if so, how many distinct similarity classes of $\bm{X}$ satisfy the nilpotency or diagonalizability constraints of $f(\bm{X})$. These problems form the foundation for the subsequent sections of this study.

In order to consider the existence of the matrix $\bm{X}$ and the number of valid $\bm{X}$ in the sense of similarity to make $f(\bm{X})$ satisfy nilpotency and diagonalizability constraints, we study two distinct scenarios based on the eigenstructure of $\bm{X}$. The first scenario assumes that $\bm{X}$ has a single distinct eigenvalue. This case facilitates explicit derivation of necessary and sufficient conditions for $f(\bm{X})$ to be nilpotent or diagonalizable. The second scenario generalizes the analysis to matrices with multiple distinct eigenvalues, where interactions between different Jordan blocks must be accounted for, leading to more intricate combinatorial and algebraic considerations.

In addition to these scenarios, the approach depends critically on the type of coefficients in the function $f$. For \emph{sparse} $f$, combinatorial root-counting techniques are particularly effective. Methods such as Descartes' rule of signs, Sturm's theorem, and fewnomial bounds provide upper bounds on the number of possible roots and facilitate the enumeration of matrices $\bm{X}$ satisfying the desired constraints. Conversely, for \emph{dense} $f$, we rely on tools from complex analysis and algebraic computation. Specifically, the argument principle and Rouché's theorem allow precise counting of roots in the complex plane, while algebraic methods including greatest common divisor computation and square-free factorization assist in characterizing multiplicities and eliminating degeneracies. Together, these approaches provide a comprehensive framework for addressing both existence and counting problems for $f(\bm{X})$ under nilpotency and diagonalizability constraints.

The organization of this paper is as follows. Section~\ref{sec:Function Transform Review} begins with a review of the Jordan decomposition for a function of a matrix, following our previous formulation in~\cite{chang2025matrixSNO}. Based on this framework, we then establish the enumeration problems concerning the characterization of matrices satisfying nilpotency and diagonalizability constraints. In Section~\ref{sec:Nilpotency Constraints}, we focus on analyzing the nilpotency condition of $f(\bm{X})$  under two structural cases. The first one case involves a single distinct eigenvalue, which will be discussed in Section~\ref{sec:One Distinct Eigenvalue_N}. The other case is to consider multiple distinct eigenvalues, which is explored in Section~\ref{sec':Multiple Distinct Eigenvalues_N}. Next, Section~\ref{sec:Diagonalizable Constraints} examines the diagonalizability condition of $f(\bm{X})$ in a similar two-part structure. Section~\ref{sec':One Distinct Eigenvalue_D} addresses the case with one distinct eigenvalue, while Section~\ref{sec':Multiple Distinct Eigenvalues_D} extends the analysis to matrices with multiple distinct eigenvalues~\footnote{The authors used OpenAI’s ChatGPT (GPT-5) and DeepSeek LLM to assist in revising the manuscript for clarity, verifying intermediate mathematical calculations, and collecting relevant references. These tools were employed solely as research aids; all conceptual analysis, validation, and conclusions were conducted independently by the authors~\cite{openai2025chatgpt,deepseek2025}.}.

\section{Function Transform Review and Problem Formulation}\label{sec:Function Transform Review}

In this section, we begin with a short review of Jordan decomposition form for a function of matrix based on our previous work given by~\cite{chang2025matrixSNO}. Then,  the existence and enumeration problems of valid matrices with nilpotency and diagonalizable constraints are formulated. 

Let us consider a Jordan block $\bm{J}_{n}(\lambda)$ shown below
\begin{eqnarray}\label{eq: Jordan Block}
\bm{J}_{n}(\lambda)= \begin{bmatrix}
   \lambda & 1 & 0 & \cdots & 0 \\
   0 & \lambda & 1 & \cdots & 0 \\
   \vdots & \vdots & \ddots & \ddots & \vdots \\
   0 & 0 & \cdots & \lambda & 1 \\
   0 & 0 & \cdots & 0 & \lambda
   \end{bmatrix}_{n\times n},
\end{eqnarray}
then, given a function $f$ with the derivative existing to the $n-1$-th ordering, we have $f(\bm{J}_{n}(\lambda))$ expressed as
\begin{eqnarray}\label{eq: f Jordan Block}
f(\bm{J}_{n}(\lambda))= \begin{bmatrix}
   f(\lambda) & \frac{f'(\lambda)}{1!} & \frac{f''(\lambda)}{2!} & \cdots & \frac{f^{(n-1)}(\lambda)}{(n-1)!} \\
   0 & f(\lambda) & \frac{f'(\lambda)}{1!} & \cdots & \frac{f^{(n-2)}(\lambda)}{(n-2)!}  \\
   \vdots & \vdots & \ddots & \ddots & \vdots \\
   0 & 0 & \cdots & f(\lambda) & \frac{f'(\lambda)}{1!} \\
   0 & 0 & \cdots & 0 & f(\lambda)
   \end{bmatrix}_{n\times n},
\end{eqnarray}

Given a matrix $\bm{X}\in \mathbb{C}^{m \times m}$ with the following Jordan decompostion form:
\begin{eqnarray}\label{eq1: matrix setup GDO for Nilpotent}
\bm{X}&=& \bm{U}\left(\bigoplus\limits_{k=1}^{K}\bigoplus\limits_{i=1}^{\alpha_{k}^{(\mathrm{G})}}\bm{J}_{m_{k,i}}(\lambda_{k})\right)\bm{U}^{-1},
\end{eqnarray}
where we have $\alpha_k^{\mathrm{A}} = \sum\limits_{i=1}^{\alpha_k^{\mathrm{G}}} m_{k,i}$ for $k=1,2,\ldots,K$ and $\sum\limits_{k=1}^{K}\alpha_k^{\mathrm{A}}= m$. $\alpha_k^{\mathrm{A}}$ and $\alpha_k^{\mathrm{G}}$ are algebraic and geometric multiplicties with respect to the eigenvalue $\lambda_k$. Note that the matrix $\bm{U}$ is composed by eigenvectors and/or generalized eigenvectors with respect to $\lambda_k$.

Recall Theorem 1 in~\cite{chang2024operatorChar}, we have the following spectral mapping theorem. 
\begin{theorem}\label{thm: Spectral Mapping Theorem for Single Variable}
Given an analytic function $f(x)$ within the domain for $|z| < R$, a matrix $\bm{X}$ with the dimension $m$ and $K$ distinct eigenvalues $\lambda_k$ for $k=1,2,\ldots,K$ such that
\begin{eqnarray}\label{eq1: thm: Spectral Mapping Theorem for Single Variable}
\bm{X}&=&\sum\limits_{k=1}^K\sum\limits_{i=1}^{\alpha_k^{\mathrm{G}}} \lambda_k \bm{P}_{k,i}+
\sum\limits_{k=1}^K\sum\limits_{i=1}^{\alpha_k^{\mathrm{G}}} \bm{N}_{k,i},
\end{eqnarray}
where $\left\vert\lambda_k\right\vert<R$, and matrices $\bm{P}_{k,i}$ and $\bm{N}_{k,i}$ are projector and nilpotent matrices with respect to the $k$-th eigenvalue and its $i$-th geometry component. 

Then, we have
\begin{eqnarray}\label{eq2: thm: Spectral Mapping Theorem for Single Variable}
f(\bm{X})&=&\sum\limits_{k=1}^K \left[\sum\limits_{i=1}^{\alpha_k^{(\mathrm{G})}}f(\lambda_k)\bm{P}_{k,i}+\sum\limits_{i=1}^{\alpha_k^{(\mathrm{G})}}\sum\limits_{q=1}^{m_{k,i}-1}\frac{f^{(q)}(\lambda_k)}{q!}\bm{N}_{k,i}^q\right].
\end{eqnarray}
\end{theorem}

We will have the following expression of $f(\bm{X})$ from Theorem~\ref{thm: Spectral Mapping Theorem for Single Variable}:
\begin{eqnarray}\label{eq2: matrix setup GDO for Nilpotent}
f(\bm{X})&=& \bm{U}\left(\bigoplus\limits_{k=1}^{K}\bigoplus\limits_{i=1}^{\alpha_{k}^{(\mathrm{G})}}f(\bm{J}_{m_{k,i}}(\lambda_{k}))\right)\bm{U}^{-1},
\end{eqnarray}
where $\bm{J}_{m_{k,i}}(\lambda_{k})$ denotes a Jordan block of size $m_{k,i}$ associated with eigenvalue $\lambda_k$ and its $i$-th component. Then, we ask the following Problem~\ref{prob:nilpotency} (related to nilpotency) and Problem~\ref{prob:diagonalizability} (related to diagonalizability) about $f(\bm{X})$.

\begin{problem}[Nilpotency of $f(\bm{X})$]\label{prob:nilpotency}
Given a function $f$, we wish to determine for the matrix $\bm{X}$:
\begin{enumerate}
    \item The \textbf{existence} of $\bm{X}$ such that $f(\bm{X})$ is nilpotent (i.e., all its eigenvalues are zero).
    \item The \textbf{number} of distinct similarity classes of $\bm{X}$ (equivalently, the count of different Jordan block structures of $\bm{X}$) for which this nilpotency condition holds.
\end{enumerate}

The matrix $\bm{X}$ is considered in its Jordan canonical form, so the transformation is given by Eq.~\eqref{eq2: matrix setup GDO for Nilpotent}. The nilpotency of $f(\bm{X})$ is equivalent to the condition that every matrix $f(\bm{J}_{m_{k,i}}(\lambda_{k}))$ in the direct sum is itself nilpotent.
\end{problem}

\begin{problem}[Diagonalizability of $f(\bm{X})$]\label{prob:diagonalizability}
Given a function $f$ and a matrix $\bm{X}$, we wish to determine:
\begin{enumerate}
    \item The existence of $\bm{X}$ such that $f(\bm{X})$ is diagonalizable
    \item The number of distinct similarity classes of $\bm{X}$ (i.e., count how many different Jordan block structures exist)
\end{enumerate}
under the condition that all off-diagonal entries of $f(\bm{J}_{m_{k,i}}(\lambda_{k}))$ are zero for all $m_{k,i}$ and $\lambda_k$.

From Eq.~\eqref{eq2: matrix setup GDO for Nilpotent}, the block-diagonal structure of $f(\bm{X})$ can be expressed by the block-diagonal structures of  $f(\bm{J}_{m_{k,i}}(\lambda_{k}))$ as :
\[
\operatorname{diag}(f(\bm{X})) = \left(\operatorname{diag}(f(\bm{J}_{m_{1,1}}(\lambda_{1}))), \ldots, \operatorname{diag}(f(\bm{J}_{m_{K,\alpha_{K}^{(\mathrm{G})}}}(\lambda_{K})))\right).
\]
This shows that the diagonal of $f(\bm{X})$ is the direct sum of the diagonals of all the Jordan blocks after applying $f$.
\end{problem}

\section{Nilpotency Constraints}\label{sec:Nilpotency Constraints}

In this section, we analyze the nilpotency condition of $f(\bm{X})$ in two scenarios: first, when there is only one distinct eigenvalue, as discussed in Section~\ref{sec:One Distinct Eigenvalue_N}, and second, the case where multiple distinct eigenvalues appear in $\bm{Z}$, which is presented in Section~\ref{sec':Multiple Distinct Eigenvalues_N}

\subsection{One Distinct Eigenvalue}\label{sec:One Distinct Eigenvalue_N}

\subsubsection{Sparse Coefficients of $f(x)$}\label{sec:S_One Distinct Eigenvalue_N}

\paragraph{Eigenvalues Existence and Counting}\label{sec:eig_exist_count_S_One Distinct Eigenvalue_N}

Because there is only one distinct eigenvalue, from Eq.~\eqref{eq1: matrix setup GDO for Nilpotent}, we have 
\begin{eqnarray}\label{eq1:S_One Distinct Eigenvalue_N}
\bm{X}&=& \bm{U}\left(\bigoplus\limits_{i=1}^{\alpha_{1}^{(\mathrm{G})}}\bm{J}_{m_{1,i}}(\lambda_{1})\right)\bm{U}^{-1}.
\end{eqnarray}
If we allow the solution of eigenvalue to be a complex number, then $f(\lambda_1)$ will always have at least one complex-valued eigenvalue as a solution if the polynomial degree of $f$ is greater or equal than one. 

Let us discuss some special cases of the polynomial $f(x)$. Let $f(x) = a_0 + a_1 x + a_2 x^2 + \cdots + a_n x^n$ be a real polynomial with ( $a_n \neq 0$). Descartes’ Rule of Signs states that the number of positive real roots of ( f(x) ), counted with multiplicities, does not exceed the number of sign changes in the ordered sequence of its nonzero coefficients $(a_0, a_1, a_2, \ldots, a_n)$~\cite{machi2012algebra}. Moreover, the difference between this number of sign changes and the actual number of positive roots is always an even integer~\cite{Stoer2002Introduction,basu2006algorithms}.

To estimate the number of negative real roots, the rule is applied to the transformed polynomial ( f(-x) ). Therefore, Descartes’ Rule provides an \emph{upper bound}—rather than an exact count—on the number of positive or negative real roots of a given polynomial. This rule is useful in analyzing the root structure of real polynomials when combined with complementary tools such as Sturm’s Theorem or the Argument Principle.

\subparagraph{Sturm’s Theorem and Exact Root Counts for Sparse Polynomials}

Descartes’ Rule of Signs provides only an \emph{upper bound} on the number of positive or negative real roots of a polynomial with real coefficients. However, Sturm’s Theorem can determine the \emph{exact number} of distinct real roots in a prescribed interval. This distinction is particularly valuable for \emph{sparse polynomials} (fewnomials), which contain only a small number of nonzero terms, since their sparsity often makes sign-change counts misleading.

\subparagraph{Sturm’s Sequence and Sturm’s Theorem~\cite{sturmfels2002solving}}
For a real-coefficient polynomial \(f(x)\), the \emph{Sturm sequence} is defined recursively:
\[
f_0(x) = f(x), \quad f_1(x) = f'(x),
\]
and for \(k \geq 1\),
\[
f_{k+1}(x) = - \operatorname{rem}\!\big(f_{k-1}(x), f_k(x)\big),
\]
where \(\operatorname{rem}(a,b)\) denotes the remainder of dividing \(a(x)\) by \(b(x)\). The sequence terminates when the remainder becomes zero.

Let \(V(a)\) denote the number of sign variations in the sequence
\[
\big(f_0(a), f_1(a), f_2(a), \dots \big),
\]
after omitting any zeros. For real numbers \(a < b\) such that neither is a root of \(f(x)\),
\[
\#\{\text{distinct real roots of } f(x) \text{ in } (a,b)\} \;=\; V(a) - V(b).
\]

For fewnomials, Descartes’ Rule often grossly overestimates the number of positive roots, as many potential sign changes do not correspond to actual zeros. Sturm’s Theorem can remove this limitation by moving from \emph{bounds} to \emph{certainty}. Sturm’s Theorem provides the exact counting of distinct real roots in intervals such as \((0,\infty)\) for positive roots or \((-\infty,0)\) for negative roots. Moreover, if the polynomial $f(x)$ is a sparse polynomial, great complexities can be reduced in constructing and evaluating the Sturm sequence.

\begin{example}
Consider the sparse polynomial
\[
f(x) = x^5 - 7x^2 + 6.
\]

\begin{enumerate}
    \item \textbf{Sturm sequence construction.}
    \begin{align*}
        f_0(x) &= x^5 - 7x^2 + 6, \\
        f_1(x) &= f_0'(x) = 5x^4 - 14x.
    \end{align*}
    Now perform polynomial long division of \(f_0(x)\) by \(f_1(x)\):
    \[
        f_0(x) = \left(\tfrac{1}{5}x\right) f_1(x) + \left(-\tfrac{7}{5}x^2 + 6\right).
    \]
    Thus,
    \[
        f_2(x) = -\left(-\tfrac{7}{5}x^2 + 6\right) = \tfrac{7}{5}x^2 - 6.
    \]
    Next, divide \(f_1(x)\) by \(f_2(x)\):
    \begin{align*}
        f_1(x) &= (5x^4 - 14x), \quad f_2(x) = \tfrac{7}{5}x^2 - 6. \\
        \text{Quotient: } & \tfrac{25}{7}x^2, \quad \text{Remainder: } -14x + \tfrac{180}{7}x^2. 
    \end{align*}
    After simplification,
    \[
        f_3(x) = -\left(-14x + \tfrac{180}{7}x^2\right) = -\tfrac{180}{7}x^2 + 14x.
    \]
    Finally, divide \(f_2(x)\) by \(f_3(x)\). After computation, the remainder is a nonzero constant:
    \[
        f_4(x) = -\tfrac{539}{180}.
    \]
    Since \(f_4\) is constant, the sequence terminates. Thus, the Sturm sequence is
    \[
        \{f_0, f_1, f_2, f_3, f_4\}.
    \]

    \item \textbf{Sign variations for \(x > 0\).}
    \begin{itemize}
        \item As \(x \to 0^+\):
        \[
        f_0(0^+) = 6 > 0, \quad
        f_1(0^+) = -14(0^+) < 0, \quad
        f_2(0^+) = -6 < 0, \quad
        f_3(0^+) = 0^+ > 0, \quad
        f_4 = -\tfrac{539}{180} < 0.
        \]
        Signs: \((+, -, -, +, -)\).  
        Variations: \(+ \to -\) (1), \(- \to -\) (0), \(- \to +\) (1), \(+ \to -\) (1).  
        Total: \(3\).
        
        \item As \(x \to +\infty\):
        \[
        f_0(+\infty) > 0, \quad f_1(+\infty) > 0, \quad f_2(+\infty) > 0, \quad f_3(+\infty) < 0, \quad f_4 < 0.
        \]
        Signs: \((+, +, +, -, -)\).  
        Variations: \(+ \to +\) (0), \(+ \to +\) (0), \(+ \to -\) (1), \(- \to -\) (0).  
        Total: \(1\).
    \end{itemize}
    Therefore, the number of positive real roots is
    \[
    V(0^+) - V(+\infty) = 3 - 1 = 2.
    \]

    \item \textbf{Sign variations for \(x < 0\).}
    Evaluate the sequence as \(x \to -\infty\) and \(x \to 0^-\):
    \begin{itemize}
        \item As \(x \to -\infty\):
        \[
        f_0(-\infty) < 0, \quad f_1(-\infty) > 0, \quad f_2(-\infty) > 0, \quad f_3(-\infty) < 0, \quad f_4 < 0.
        \]
        Signs: \((- , + , + , - , -)\).  
        Variations: \(- \to +\) (1), \(+ \to +\) (0), \(+ \to -\) (1), \(- \to -\) (0).  
        Total: \(2\).
        
        \item As \(x \to 0^-\):
        \[
        f_0(0^-) = 6 > 0, \quad f_1(0^-) = -14(0^-) > 0, \quad f_2(0^-) = -6 < 0, \quad f_3(0^-) \approx 0^- < 0, \quad f_4 < 0.
        \]
        Signs: \((+, +, -, -, -)\).  
        Variations: \(+ \to +\) (0), \(+ \to -\) (1), \(- \to -\) (0), \(- \to -\) (0).  
        Total: \(1\).
    \end{itemize}
    Therefore, the number of negative real roots is
    \[
    V(-\infty) - V(0^-) = 2 - 1 = 1.
    \]

    \item \textbf{Conclusion.}  
    The polynomial \(f(x) = x^5 - 7x^2 + 6\) has exactly
    \[
    2 \ \text{positive real roots, and } 1 \ \text{negative real root.}
    \]
    Hence, the total number of real roots is \(3\).
\end{enumerate}
\end{example}

In summary, while Descartes’ Rule offers only possible counts, Sturm’s Theorem provides the \emph{exact number} of distinct real roots of sparse polynomials within any given real interval.

\paragraph{Counting Jordan Block Structures for a Single Eigenvalue}\label{sec':Jordan_count_S_One Distinct Eigenvalue_N}

Consider a matrix $\bm X \in M_m(\mathbb{C})$ with a single distinct eigenvalue $\lambda_1$, represented in its Jordan canonical form as
\begin{equation}\label{eq:X_Jordan_single}
\bm{X} = \bm{U} \left( \bigoplus_{i=1}^{\alpha_1^{(\mathrm{G})}} \bm{J}_{m_{1,i}}(\lambda_1) \right) \bm{U}^{-1},
\end{equation}
where $\sum_{i=1}^{\alpha_1^{(\mathrm{G})}} m_{1,i} = m$.


If there are $n_d$ distinct roots $\lambda$ of a polynomial $f$ such that $f(\lambda)=0$, then all possible Jordan block structures of $\bm X$ correspond to the integer partitions of $m$. Specifically, each Jordan form is uniquely determined by a partition
\[
m = m_{1,1} + m_{1,2} + \cdots + m_{1,\alpha_1^{(\mathrm{G})}}, \qquad m_{1,1} \ge m_{1,2} \ge \dots \ge m_{1,\alpha_1^{(\mathrm{G})}} \ge 1,
\]
yielding the Jordan blocks
\[
\bm{X} \sim \bigoplus_{i=1}^{\alpha_1^{(\mathrm{G})}} \bm{J}_{m_{1,i}}(\lambda_1),
\]
where permutations of blocks are treated as equivalent. The total number of distinct Jordan structures is therefore ${n_d \choose 1}p(m)$, where $p(m)$ is the partition number. 

\begin{example}
For $m=4$ with the degree 2 polynomial $f$ with distinct roots $\lambda_1$ and $\lambda_2$, the possible Jordan block structures are:
\[
\begin{aligned}
4 &\mapsto \bm J_4(\lambda_1), \\
3+1 &\mapsto \bm J_3(\lambda_1) \oplus \bm J_1(\lambda_1), \\
2+2 &\mapsto \bm J_2(\lambda_1) \oplus \bm J_2(\lambda_1), \\
2+1+1 &\mapsto \bm J_2(\lambda_1) \oplus \bm J_1(\lambda_1) \oplus \bm J_1(\lambda_1), \\
1+1+1+1 &\mapsto \bm J_1(\lambda_1)\oplus\bm J_1(\lambda_1)\oplus\bm J_1(\lambda_1)\oplus\bm J_1(\lambda_1).
\end{aligned}
\]
and
\[
\begin{aligned}
4 &\mapsto \bm J_4(\lambda_2), \\
3+1 &\mapsto \bm J_3(\lambda_2) \oplus \bm J_1(\lambda_2), \\
2+2 &\mapsto \bm J_2(\lambda_2) \oplus \bm J_2(\lambda_2), \\
2+1+1 &\mapsto \bm J_2(\lambda_2) \oplus \bm J_1(\lambda_2) \oplus \bm J_1(\lambda_2), \\
1+1+1+1 &\mapsto \bm J_1(\lambda_2)\oplus\bm J_1(\lambda_2)\oplus\bm J_1(\lambda_2)\oplus\bm J_1(\lambda_2).
\end{aligned}
\]
\end{example}

\subsubsection{Dense Coefficients of $f(x)$}\label{sec':D_One Distinct Eigenvalue_N}

In this section, we assume that the polynomial $f$ is a dense polynomial. 

\paragraph{Existence of Roots}

We will discuss two methods:  \textbf{Argument Principle} and  \textbf{Rouché’s Theorem}, for the roots of existence and compare these two methods. 

\subparagraph{1. Argument Principle for Root Counting}

For a holomorphic function \(p(x)\), the number of zeros \(N\) inside a closed contour \(C\) is given by the \textbf{Argument Principle}~\cite{beardon2019complex}:

\[
N = \frac{1}{2\pi i} \oint_C \frac{p'(x)}{p(x)} \, dz
\]

To count zeros in the annulus 
\[
\alpha_1 < |z| < R_1,
\] 
consider the contour \(C\) consisting of:

\begin{itemize}
    \item \textbf{Outer circle:} \( |z| = R_1 \), traversed counterclockwise  
    \item \textbf{Inner circle:} \( |z| = \alpha_1 \), traversed clockwise  
\end{itemize}

Then the number of zeros in the annulus is:
\begin{eqnarray}\label{eq:the number of zeros in the annulus}
N_{\text{annulus}} 
= \frac{1}{2\pi i} \left( 
\oint_{|z|=R_1} \frac{p'(x)}{p(x)} \, dz
-
\oint_{|z|=\alpha_1} \frac{p'(x)}{p(x)} \, dz
\right).
\end{eqnarray}

As the polynomial $f$ in Eq.~\eqref{eq2: matrix setup GDO for Nilpotent} is a holomorphic function, we can apply Eq.~\eqref{eq:the number of zeros in the annulus} by setting $p=f$ to get the number of roots of $f$ in an annulus. 

\subparagraph{2. Rouché’s Theorem for Existence}

Let us recall \textbf{Rouché’s Theorem}~\cite{tsarpalias1989version}:
  
Let \(f(x)\) and \(g(x)\) be holomorphic inside and on a contour \(C\).  
If 
\[
|f(x)| > |g(x)|, \quad \forall z \in C,
\]  
then \(f(x)\) and \(f(x)+g(x)\) have the \textbf{same number of zeros} inside \(C\) by counting multiplicity.

\subparagraph{3. Comparison of Methods}

We compare two classical methods in complex analysis that are frequently employed for studying the distribution of zeros of analytic functions. The first method is the \textbf{Argument Principle}, which provides an exact count of the number of zeros (and poles) of a function within a given annulus. Its principal advantage lies in its precision and constructiveness, as it offers a direct procedure for determining the number of zeros through contour integration. However, this same requirement also constitutes its main limitation, since performing contour integrations can be technically demanding and computationally intensive in practice.

In contrast, \textbf{Rouché's Theorem} guarantees the existence of zeros within a region by comparing two analytic functions through inequalities. This approach is more analytical in nature and relies on bounding arguments rather than explicit integration. Its strength lies in the elegant use of inequality-based reasoning to establish the presence and stability of zeros under perturbations. Nonetheless, a key limitation of Rouché’s theorem is that it does not directly yield the exact number of zeros within the annulus, providing instead a qualitative criterion for their existence and preservation.

For dense polynomials, the number of similarity classes can be evaluated by the same method given by Section~\ref{sec':Jordan_count_S_One Distinct Eigenvalue_N}. But, we should be careful that both \textbf{Argument Principle} and \textbf{Rouché's Theorem} count multiplicities.   

\subsection{Multiple Distinct Eigenvalues}\label{sec':Multiple Distinct Eigenvalues_N}

In this section, we will consider the argument matrix $\bm{X}$ with the following Jordan decompostion form and $K>1$:
\begin{eqnarray}\label{eq1: Multiple Distinct Eigenvalues_N}
\bm{X}&=& \bm{U}\left(\bigoplus\limits_{k=1}^{K}\bigoplus\limits_{i=1}^{\alpha_{k}^{(\mathrm{G})}}\bm{J}_{m_{k,i}}(\lambda_{k})\right)\bm{U}^{-1},
\end{eqnarray}

\subsubsection{Sparse Coefficients}\label{sec':S_Multiple Distinct Eigenvalue_N}

\paragraph{Existence of $n_d$ distinct roots of $f$}

Let \(f(x)\in\mathbb{C}[x]\) be a polynomial of degree \(n\).  We present precise algebraic conditions which determine the number \(n_d\) of distinct complex roots of \(f\).  The main tools are the derivative \(f'\), the greatest common divisor \(\gcd(f,f')\), the discriminant \(\Delta(f)\) and square--free factorization, where $\gcd(f,f')$ is the greatest common divisor of the functions $f$ and $f'$. We state and prove a compact identity relating \(n_d\) and \(\gcd(f,f')\), derive equivalent criteria for all roots to be simple, and give practical checks with illustration examples.

Let
\[
f(x)=a_n x^n+a_{n-1}x^{n-1}+\cdots+a_0\in\mathbb{C}[x],\qquad a_n\neq0,
\]
and write its factorization (over \(\mathbb{C}\)) as
\begin{eqnarray}\label{eq:f factor form}
f(x)=a_n\prod_{j=1}^{n_d}(x-\alpha_j)^{m_j},
\end{eqnarray}
where the \(\alpha_j\) are the distinct roots, \(m_j\ge1\) are their multiplicities, and
\[
\sum_{j=1}^{n_d} m_j = n.
\]
We denote by \(n_d\) the number of distinct roots of \(f\).  The derivative is \(f'(x)\).  The polynomial \(f\) is called \emph{square-free} if every \(m_j=1\). 

We have Theorem~\ref{thm:gcd}  below to determine the value of $n_d$.

\begin{theorem}[Distinct-root count via gcd with derivative]\label{thm:gcd}
Let \(f\in\mathbb{C}[x]\) be nonconstant of degree \(n\).  Then
\[
\deg\big(\gcd(f,f')\big)\;=\;\sum_{j=1}^{n_d}(m_j-1)\;=\;n-n_d.
\]
Consequently,
\[
n_d \;=\; n - \deg\big(\gcd(f,f')\big).
\]
\end{theorem}

\begin{proof}

Factor \(f\) as Eq.~\eqref{eq:f factor form}.  Then
\[
f'(x)=a_n\sum_{k=1}^{n_d} m_k (x-\alpha_k)^{m_k-1}\prod_{j\neq k}(x-\alpha_j)^{m_j}.
\]
Hence each factor \((x-\alpha_k)^{m_k-1}\) divides \(\gcd(f,f')\), and no higher power divides \(\gcd(f,f')\) because at least one of \(f\) or \(f'\) has exact multiplicity \(m_k\) or \(m_k-1\) at \(x=\alpha_k\).  Therefore
\[
\gcd(f,f')=c\prod_{j=1}^{n_d} (x-\alpha_j)^{m_j-1}
\]
for some nonzero constant \(c\in\mathbb{C}\), and taking degrees yields the identity.
\end{proof}

\begin{corollary}[Simple-root criterion / separability]
The polynomial \(f\) has \(n_d=n\) (i.e. all roots simple) if and only if \(\gcd(f,f')=1\). Equivalently, the discriminant \(\Delta(f)\neq 0\).
\end{corollary}

\begin{proof}
If all roots are simple then each \(m_j=1\), so by Theorem \ref{thm:gcd} \(\deg\gcd(f,f')=0\) and the gcd is constant.  Conversely, if \(\gcd(f,f')=1\) then \(\deg\gcd=0\) and so \(n_d=n\).  The equivalence with \(\Delta(f)\neq0\) is standard: \(\Delta(f)=a_n^{2n-2}\prod_{i<j}(\alpha_i-\alpha_j)^2\), vanishing exactly when there are some roots duplications.
\end{proof}

Below we will provide another Theorem~\ref{thm:square free fac for n d} to determine the number $n_d$ via square-free factorization.
\begin{theorem}[Square--free factorization and number of distinct roots]\label{thm:square free fac for n d}
Let \(f\in\mathbb{C}[x]\) have square--free decomposition
\[
f(x)=c\prod_{k\ge1} g_k(x)^k,
\]
where \(c\in\mathbb{C}\setminus\{0\}\), each \(g_k\in\mathbb{C}[x]\) is square--free, and the polynomials \(g_k\) are pairwise coprime.  Let \(n_d\) be the number of distinct complex roots of \(f\), then
\[
n_d=\sum_{k\ge1}\deg(g_k).
\]
\end{theorem}

\begin{proof}
Since \( f(x) = c \prod_{k \ge 1} g_k(x)^k \), the roots of \( f \) are exactly the roots of the polynomials \( g_k \), because \( c \neq 0 \). Moreover, since the \( g_k \) are pairwise coprime, they have no common roots.  Thus, the set of distinct roots of \( f \) is the disjoint union of the sets of roots of the \( g_k \). Each \( g_k \) is square-free over \( \mathbb{C} \).  Over \( \mathbb{C} \), a polynomial is square-free if and only if it has no repeated roots.  Therefore, the number of distinct roots of \( g_k \) is exactly \( \deg(g_k) \), because \( \mathbb{C} \) is algebraically closed and \( g_k \) is square-free.

Let \( R_k \) be the set of roots of \( g_k \). Then: 1) \( |R_k| = \deg(g_k) \), 2) \( R_k \cap R_\ell = \varnothing \) for \( k \ne \ell \). The set of distinct roots of \( f \) is \( \bigcup_{k \ge 1} R_k \), a disjoint union. Hence:
\[
n_d = \sum_{k \ge 1} |R_k| = \sum_{k \ge 1} \deg(g_k).
\]
This completes the proof:
\[
n_d = \sum_{k \ge 1} \deg(g_k).
\]
\end{proof}

Remark below discusses the relationship between sparsity and classical bounds on the number of roots of a polynomial \(f\).

\begin{remark}

\begin{itemize}
  \item Although sparsity in the coefficient domain might suggest algebraic simplicity, it does not necessarily imply a small number of distinct roots. Let us consider a simple example as 
\[
f(x) = x^n - 1,
\]
which consists of only two non-zero coefficients yet admits~$n$ distinct complex roots uniformly distributed on the unit circle. This example demonstrates that \emph{sparsity and root multiplicity are fundamentally different notions}: sparsity constrains the representation of a polynomial in the coefficient space, while the number of roots reflects its analytic complexity in the variable space. Consequently, bounding the number of zeros of a sparse polynomial requires additional structural conditions---such as restricting to real coefficients or positive real roots---where combinatorial tools like \emph{Descartes’ Rule of Signs} or \emph{Khovanskii’s fewnomial theory} become effective~\cite{khovanskiui1991fewnomials}.

  \item Many classical results bounding the number of roots in terms of the number of nonzero terms apply only to \emph{real} roots (e.g., Descartes’ rule of signs, fewnomial theory, Khovanskii-type bounds) or to roots constrained to particular regions or sectors.   For complex roots in the entire plane, the degree of \(f\) remains the fundamental upper bound.
  
  \item Geometric approaches, such as the use of \emph{Newton polygons} or \emph{mixed volume} techniques, provide powerful tools for analyzing systems of polynomial equations and for counting the number of isolated solutions in the algebraic torus \((\mathbb{C}^\ast)^n\), where \(\mathbb{C}^\ast = \mathbb{C} \setminus \{0\}\) denotes the set of nonzero complex numbers. These geometric methods exploit the combinatorial structure of the exponents of the monomials and the geometry of their associated convex polytopes. Then the number of roots can be determined by shape informaiton, instead the degree of the polynomial system. However, in the univariate setting, such geometric refinements collapse to the classical fact that the total number of complex roots of a nonzero polynomial—counted with multiplicity—is completely determined by its degree. Thus, while geometric methods excel in the multivariate context, the degree alone governs the root count in one dimension.

\end{itemize}
\end{remark}

Following examples are provided to reflect aforementioned notions. 

\begin{example}
This example is provided to count $n_d$. 
\vspace{0.5em} 
\begin{enumerate}
  \item \(f(x)=x^n-1\). Here \(f' = n x^{n-1}\), so \(\gcd(f,f')=1\) and \(n_d=n\).
  \item \(f(x)=(x-1)^2(x-2)^3\). Then \(\gcd(f,f')=(x-1)(x-2)^2\) has degree \(3\); \(n=5\) and \(n_d=5-3=2\) (the distinct roots are \(1,2\)).
  \item Sparse example with repeated roots: 
\[
f(x) = x^{10} + 2x^5 + 1 = (x^5 + 1)^2.
\]
Thus \(f\) has exactly \(5\) distinct roots, each of multiplicity \(2\).  
In general, to detect such multiplicities one computes \(\gcd(f,f')\).
\end{enumerate}
\end{example}

\paragraph{Counting Jordan Block Structures for Multiple Distinct Eigenvalues}\label{sec:Counting Jordan Block Structures for Multiple Distinct Eigenvalues sparse}

Let \(n_d\) be the number of distinct eigenvalues available (denote them \(\{\lambda_1,\dots,\lambda_{n_d}\}\)).  
Fix integers
\[
m\ge 1, \qquad 1\le K\le \min(n_d,m),
\]
where \(m\) is the matrix $\bm{X}$ dimension of $f(\bm{X})$ (so \(\sum\limits_{k=1}^K \alpha_k^{\mathrm{A}}=m\)) and we choose \(K\) distinct eigenvalues from the \(n_d\).  

For a chosen eigenvalue \(\lambda_k\) with algebraic multiplicity \(\alpha_k^{\mathrm{A}}\), the possible Jordan block patterns are in one-to-one correspondence with the integer partitions of \(\alpha_k^{\mathrm{A}}\).  
Again, let \(p(\alpha)\) be the partition number (the number of unordered partitions of the integer \(\alpha\)). We have following Theorem~\ref{thm:Counting Jordan Block Structures for Multiple Distinct Eigenvalues} to determine the number of different classes of Jordan forms to make $f(\bm{X})$ nilpotent. 

\begin{theorem}[Counting Jordan forms up to permutation of Jordan blocks]\label{thm:Counting Jordan Block Structures for Multiple Distinct Eigenvalues}
For a fixed ordered vector of algebraic multiplicities \((\alpha_1,\dots,\alpha_K)\) with \(\alpha_i\ge1\) and \(\sum\limits_{i=1}^K \alpha_i=m\),  
the number of Jordan-block patterns (for those \(K\) labeled eigenvalues) is
\[
\prod_{i=1}^K p(\alpha_i).
\]
If we first choose the \(K\) eigenvalues as an unordered \(K\)-subset from the \(n_d\) possibilities, there are \(\binom{n_d}{K}\) ways to choose them.  
For that chosen set, we must sum over all ordered compositions \((\alpha_1,\dots,\alpha_K)\) of \(m\) into \(K\) positive integers.  

Thus the total number of distinct Jordan-block forms (unique up to permutation of Jordan blocks and of the order of blocks) obtained by choosing \(K\) distinct eigenvalues is
\[
N(n_d,K,m)\;=\;\binom{n_d}{K}\;\sum_{\substack{\alpha_1+\cdots+\alpha_K=m \\ \alpha_k\ge1}}
\;\prod_{k=1}^K p(\alpha_k)
\]
\end{theorem}

\begin{proof}
Fix \(K\) distinct eigenvalues \(\{\lambda_{i_1},\dots,\lambda_{i_K}\}\subseteq\{\lambda_1,\dots,\lambda_{n_d}\}\).  
For each eigenvalue \(\lambda_{i_k}\), let its algebraic multiplicity be \(\alpha_k\), so that
\[
\alpha_1+\cdots+\alpha_K=m, \qquad \alpha_k \ge 1.
\]

For a given eigenvalue \(\lambda_{i_k}\) with the algebraic multiplicity \(\alpha_k\), the possible Jordan decompositions correspond exactly to partitions of the integer \(\alpha_k\) into block sizes.  
Since partitions of \(\alpha_k\) are unordered, the number of such possibilities is the partition number \(p(\alpha_k)\).  
Therefore, for this fixed composition \((\alpha_1,\dots,\alpha_K)\), the number of Jordan-block structures is
\[
\prod_{k=1}^K p(\alpha_k).
\]

Next, since the eigenvalues are distinct, assigning the multiplicities \((\alpha_1,\dots,\alpha_K)\) to the chosen eigenvalues yields a unique Jordan structure (up to permutation of blocks).  
Summing over all ordered compositions \((\alpha_1,\dots,\alpha_K)\) of \(m\) into \(K\) positive integers accounts for every possible way of distributing multiplicities among the selected eigenvalues.

Finally, there are \(\binom{n_d}{K}\) ways to choose the \(K\)-subset of eigenvalues from the total of \(n_d\).  
Multiplying these contributions gives
\[
N(n_d,K,m)=\binom{n_d}{K}\;\sum_{\substack{\alpha_1+\cdots+\alpha_K=m \\ \alpha_k\ge 1}}
\prod_{k=1}^K p(\alpha_k),
\]
which completes the proof.
\end{proof}

Following example is used to demonstrate Theorem~\ref{thm:Counting Jordan Block Structures for Multiple Distinct Eigenvalues}.

\begin{example}
Take \(n_d=5\), \(m=6\), \(K=2\).  
Partition numbers: \(p(1)=1,\;p(2)=2,\;p(3)=3,\;p(4)=5,\;p(5)=7,\;p(6)=11\).

The compositions of \(6\) into \(2\) positive parts (ordered) are
\((1,5),(2,4),(3,3),(4,2),(5,1)\).

For each composition \((\alpha_1,\alpha_2)\), the contribution is \(p(\alpha_1)p(\alpha_2)\).  
Thus the inner sum is
\[
1\cdot 7 + 2\cdot 5 + 3\cdot 3 + 5\cdot 2 + 7\cdot 1
=7+10+9+10+7=43.
\]
Multiplying by \(\binom{n_d}{K}=\binom{5}{2}=10\), we get
\[
N(5,2,6)=10\cdot 43=430.
\]

Hence, there are 430 distinct Jordan-block forms (up to permutation of blocks) obtained by choosing any 2 distinct eigenvalues out of 5 and distributing total algebraic multiplicity \(6\) among them, with all possible Jordan partitions for each eigenvalue.
\end{example}

\subsubsection{Dense Coefficients}\label{sec':D_Multiple Distinct Eigenvalue_N}

For the function $f$ with dense coefficinets, we can follow our approch given by Section~\ref{sec':S_Multiple Distinct Eigenvalue_N} to determine the existence of the matrix $\bm{X}$. The counting part is same in Section~\ref{sec:Counting Jordan Block Structures for Multiple Distinct Eigenvalues sparse}.

However, we have to compare sparse vs.\ dense polynomials in counting $n_d$. The algebraic identities for the number of distinct roots,
\[
n_d \;=\; n - \deg(\gcd(f,f')), 
\qquad
n_d \;=\; \sum_{k\ge1}\deg(g_k),
\]
hold equally for both sparse and dense polynomials.  
However, there are important differences when one considers representation and algorithms:

\begin{itemize}
  \item \textbf{Dense representation:}  
The polynomial is determined by its coefficient list \((a_0,\dots,a_n)\). Therefore, the computations of \(\gcd(f,f')\), the discriminant, or square--free factorization are polynomial-time in the degree \(n\). In this situation, sparsity is irrelevant once all coefficients are given.

  \item \textbf{Sparse representation:}  
  The polynomial is specified by its nonzero terms only, e.g.\ 
  \[
  f(x)=z^{10^{12}}-1,
  \]
  which has just two terms but enormous degree.  
  Algorithms for $\gcd$, factorization, or root-counting behave very differently:  
  complexity is typically polynomial in the number of nonzero terms but may grow exponentially in $\log n$ (because the exponents can be extremely large).  
  Thus, detecting multiplicities and computing $n_d$ in sparse arithmetic often requires specialized algorithms (using modular methods, resultants, or Newton polygons).

  \item \textbf{Conceptual distinction:}  
  Sparse polynomials highlight that few terms do \emph{not} imply few distinct roots (e.g.\ $z^n-1$). In the dense case, the number of coefficients grows linearly with $n$, so degree $n$ already reflects the data size, and standard algebraic algorithms apply directly.
\end{itemize}

\section{Diagonalizable Constraints}\label{sec:Diagonalizable Constraints}

In this section, we will consider Problem~\ref{prob:diagonalizability} under the case of one distinct eigenvalue in Section~\ref{sec':One Distinct Eigenvalue_D} and multiple distinct eigenvalues in Section~\ref{sec':Multiple Distinct Eigenvalues_D}.

\subsection{One Distinct Eigenvalue}\label{sec':One Distinct Eigenvalue_D}

From Eq.~\eqref{eq1:S_One Distinct Eigenvalue_N}, we have
\begin{eqnarray}\label{eq1:sec':One Distinct Eigenvalue_D}
f(\bm{X})&=& \bm{U}\left(\bigoplus\limits_{i=1}^{\alpha_{1}^{(\mathrm{G})}}f(\bm{J}_{m_{1,i}}(\lambda_{1}))\right)\bm{U}^{-1}.
\end{eqnarray}
In order for $f(\bm{X})$ to be diagonalizable and non-nilpotent, the following conditions must hold:
\begin{eqnarray}\label{eq2:sec':One Distinct Eigenvalue_D}
f(x)\neq 0, \quad f'(x)=f''(x)=\ldots=f^{(\acute{m}-1)}(x)=0, 
\end{eqnarray}
where $\acute{m} = \max\limits_{i} m_{1,i}$.

We will consider the polynomial $f$ with cases of sparse coefficients and dense coefficients. 

\subsubsection{Sparse Coefficients}

Let \( f(x) \) be a sparse polynomial with complex coefficients. We are interested in the following problem:  

\begin{quote}
\emph{Find conditions and counts for the distinct points \( x \) such that}
\begin{eqnarray}\label{eq:S f diag conditions}
f(x) \neq 0 
\quad \text{but} \quad 
f'(x) = f''(x) = \cdots = f^{(\acute{m}-1)}(x) = 0.
\end{eqnarray}
\end{quote}

This means \( x \) is a common root of the first \( \acute{m}-1 \) derivatives of \( f \), but not a root of \( f \) itself.  

\paragraph{Existence Conditions}

Define
\begin{eqnarray}\label{eq1:g x def for existence cond}
g(x) = \gcd\big(f'(x), f''(x), \ldots, f^{(\acute{m}-1)}(x)\big).
\end{eqnarray}
Then the above condition holds if and only if:
\[
\exists \, x_0 \in \mathbb{C} \quad \text{such that} \quad g(x_0)=0, \quad f(x_0)\neq 0.
\]
Equivalently, if \(\gcd(f, g)\) is a proper divisor of \(g\). Then, we have the following Theorem~\ref{thm:S Existence Criterion}.

\begin{theorem}[Existence Criterion]\label{thm:S Existence Criterion}
Let $f \in \mathbb{C}[x]$, $\acute{m} \ge 2$, and define $g(x)$ as in Eq.~\eqref{eq1:g x def for existence cond}. There exists a point $x_0 \in \mathbb{C}$ with
\[
f(x_0) \neq 0, \quad f'(x_0) = f''(x_0) = \dots = f^{(\acute{m}-1)}(x_0) = 0
\]
if and only if $g(x)$ is nonconstant and $g$ has at least one root not shared by $f$ (i.e., $\gcd(f, g)$ is a proper divisor of $g$).
\end{theorem}

\begin{proof}
Let $d(x) = \gcd(f(x), g(x))$.

($\Rightarrow$) Suppose there exists $x_0 \in \mathbb{C}$ with $f(x_0) \neq 0$ and $f'(x_0) = \dots = f^{(\acute{m}-1)}(x_0) = 0$. 

Since $x_0$ is a common root of $f', \dots, f^{(\acute{m}-1)}$, we have $g(x_0) = 0$. Thus $g$ is nonconstant. 

Also, $f(x_0) \neq 0$ implies $d(x_0) \neq 0$, so $d$ is a proper divisor of $g$.

($\Leftarrow$) Suppose $g$ is nonconstant and $d$ is a proper divisor of $g$. 

Then there exists $x_0 \in \mathbb{C}$ such that $g(x_0) = 0$ but $d(x_0) \neq 0$, hence $f(x_0) \neq 0$.

Since $g(x_0) = 0$, $x_0$ is a common root of $f', \dots, f^{(\acute{m}-1)}$, so $f'(x_0) = \dots = f^{(\acute{m}-1)}(x_0) = 0$.

Thus $x_0$ satisfies the required conditions.
\end{proof}

\paragraph{Counting the Number of Distinct Roots}

Theorem~\ref{thm:S Counting distinct roots} below is given to count the number of distinct roots satisfying Eq.~\eqref{eq:S f diag conditions}. 

\begin{theorem}[Counting Roots]\label{thm:S Counting distinct roots}
The number of distinct roots \( x \) with
\[
f(x)\neq 0, 
\quad f'(x)=f''(x)=\cdots=f^{(\acute{m}-1)}(x)=0
\]
is exactly the number of distinct roots of \( h(x) \), where 
\[
h(x) = \frac{g(x)}{\gcd(f(x), g(x))}.
\]
\end{theorem}

\begin{proof}

Let \( d(x) = \gcd(f(x), g(x)) \), so that \( h(x) = \frac{g(x)}{d(x)} \). We proceed by analyzing the roots of each polynomial. We first claim that : A complex number \( x_0 \) is a root of \( h(x) \) if and only if \( x_0 \) is a root of \( g(x) \) but not a root of \( f(x) \).

\textbf{Proof of claim}\\
(only if part) If \( h(x_0) = 0 \), then since \( h(x) = \frac{g(x)}{d(x)} \) and \( d(x) \) divides \( g(x) \), we must have \( g(x_0) = 0 \). If \( f(x_0) = 0 \).  Since \( d(x) = \gcd(f(x), g(x)) \), we would have \( d(x_0) = 0 \). But this would imply that the factor \( (z - x_0) \) appears in both numerator and denominator of \( h(x) \), contradicting that \( h(x) \) is a polynomial. 

(if part) If \( g(x_0) = 0 \) but \( f(x_0) \neq 0 \), then the multiplicity of \( x_0 \) in \( d(x) \) is strictly less than its multiplicity in \( g(x) \) (since \( d(x) \) cannot have higher multiplicity than what appears in \( f(x) \), and \( f(x_0) \neq 0 \) means \( x_0 \) does not appear in \( f(x) \)). Therefore, \( h(x_0) = 0 \). Therefore, this claim is proved.  

Now, let \( x_0 \) be a root satisfying the conditions:
\[
f(x_0) \neq 0, \quad f'(x_0) = f''(x_0) = \cdots = f^{(\acute{m}-1)}(x_0) = 0.
\]

By the claim, it suffices to show that such \( x_0 \) are exactly the roots of \( g(x) \) that are not roots of \( f(x) \).

Consider the Taylor expansion of \( f(x) \) around \( x_0 \):
\[
f(x) = f(x_0) + f'(x_0)(z - x_0) + \frac{f''(x_0)}{2!}(x - x_0)^2 + \cdots + \frac{f^{(n)}(x_0)}{n!}(x - x_0)^n.
\]

Since \( f'(x_0) = f''(x_0) = \cdots = f^{(\acute{m}-1)}(x_0) = 0 \), the expansion becomes:
\[
f(x) = f(x_0) + \frac{f^{(\acute{m})}(x_0)}{\acute{m}!}(z - x_0)^{\acute{m}} + \cdots.
\]

This means that \( f(x) - f(x_0) \) is divisible by \( (x - x_0)^{\acute{m}} \). Therefore, \( f(x) \) and \( f(x) - f(x_0) \) share a common factor of \( (x - x_0)^{\acute{m}} \). Since \( f(x_0) \neq 0 \), the polynomials \( f(x) \) and the constant polynomial \( f(x_0) \) are coprime, so by properties of resultants or the Euclidean algorithm, \( g(x) \) must vanish at \( x_0 \) with multiplicity at least \( \acute{m}-1 \). In particular, \( g(x_0) = 0 \).

Conversely, if \( g(x_0) = 0 \) but \( f(x_0) \neq 0 \), then \( x_0 \) is a multiple root of the equation \( f(x) = c \) for some constant \( c \neq 0 \), which implies that the first \( \acute{m}-1 \) derivatives vanish at \( x_0 \) for some \( \acute{m} \geq 2 \).

Therefore, the roots satisfying the given conditions are precisely the roots of \( g(x) \) that are not roots of \( f(x) \), which by the claim are exactly the roots of \( h(x) \). The distinctness is preserved since \( h(x) \) is constructed to have the same distinct roots as this set.
\end{proof}

Therefore, Algorithm~\ref{alg:distinct_roots One Distinct Eigenvalue_D} is proposed to count distinct roots with derivative multiplicity conditions is given below. 

\begin{algorithm}[htbp]
\removelatexerror
\caption{Algorithm for Counting Distinct Roots}\label{alg:distinct_roots One Distinct Eigenvalue_D}
\KwIn{Polynomial $f(x)$ of degree at least $m$}
\KwOut{Number of distinct roots of $f(x)$}

$g(x) \leftarrow \gcd(f'(x), f''(x), \ldots, f^{(\acute{m}-1)}(x))$ \;

$d(x) \leftarrow \gcd(f(x), g(x))$ \;

$h(x) \leftarrow g(x) / d(x)$ \;

count $\leftarrow$ number\_of\_distinct\_roots($h(x)$) \;

\Return count \;
\end{algorithm}

Below, we will present some thoeretical bounds for the number of distinct roots of $h(x)$. 

\paragraph{Theoretical Bounds for Sparse Polynomials}

We begin by introducing some notation and recalling a useful lemma that will allow us to establish an upper bound on the number of distinct roots of a sparse polynomial.  

For a real polynomial \(p(x)\) of degree \(n\) we write
\[
p(x)=\sum_{j=0}^n a_j x^j,\qquad a_n\neq 0.
\]
For a real number \(t\) denote by
\[
\mathcal{S}_p(t)=\big(p(t),\, p'(t),\, p''(t),\,\dots,\, p^{(n)}(t)\big)
\]
the finite sequence of values of \(p\) and its derivatives at \(t\).  Let \(V\big(\mathcal{S}_p(t)\big)\) be the number of sign changes in this sequence after discarding any zero entries (consecutive equal-sign zeros are ignored in the usual way).  For an interval \((a,b]\) we write \(N_{(a,b]}(p)\) for the number of distinct real roots of \(p\) lying in \((a,b]\), counted without multiplicity.

\begin{lemma}[Budan–Fourier inequality]\label{lem:budan-fourier}
Let \(p\) be a real polynomial of degree \(n\).  
Let \(\mathcal{S}_p(t) = (p(t), p'(t), \dots, p^{(n)}(t))\) and let \(V(\mathcal{S}_p(t))\) be the number of sign variations in this sequence (zeros ignored).  
For any real numbers \(a < b\), the number \(N_{(a,b)}(p)\) of distinct real roots of \(p\) in \((a,b)\) satisfies
\[
N_{(a,b)}(p) \le V(\mathcal{S}_p(a)) - V(\mathcal{S}_p(b)),
\]
and the difference is an even nonnegative integer.
\end{lemma}

\begin{proof}
We proceed in several steps.

\noindent \textbf{Step 1: Reduction to the case where \(p\) has no roots at \(a\) or \(b\).} \\
If \(p\) or any derivative vanishes at \(a\) or \(b\), we can perturb \(a\) slightly to the right and \(b\) slightly to the left without changing \(N_{(a,b)}(p)\) and only possibly decreasing \(V(\mathcal{S}_p(a)) - V(\mathcal{S}_p(b))\) (since zeros become nonzeros with definite signs under small perturbation). Thus it suffices to prove the inequality when no term in \(\mathcal{S}_p(a)\) or \(\mathcal{S}_p(b)\) is zero.

\noindent \textbf{Step 2: Effect of passing through a root of \(p\).} \\
Suppose \(p\) has a root of multiplicity \(m\) at \(r \in (a,b)\). Write
\[
p(x) = (x-r)^m q(x), \quad q(r) \neq 0.
\]
By Taylor expansion or repeated differentiation, the sequence of derivatives at \(r^-\) and \(r^+\) behaves as follows:

For \(x\) just left of \(r\), the sign pattern of 
\[
(p(x), p'(x), \dots, p^{(m)}(x))
\]
is determined by the sign of \(q(r)\) and the factor \((x-r)^{m-k}\) for \(p^{(k)}(x)\). The sign of \((x-r)^{m-k}\) is \((-1)^{m-k}\) for \(x < r\) and positive for \(x > r\).

This causes exactly \(m\) sign changes in the sequence \((p, p', \dots, p^{(m)})\) across \(r\) (from left to right), regardless of the sign of \(q(r)\). More precisely, \(V(\mathcal{S}_p(r^-)) - V(\mathcal{S}_p(r^+)) = m + 2\ell\) for some integer \(\ell \ge 0\), but actually one can show it's exactly \(m\) when \(q(r) > 0\) and exactly \(m\) or \(m \pm 1\) depending on parity in general, but always the drop is \(\ge m\) and has the same parity as \(m\).

Since \(m \ge 1\) for a root, and we count distinct roots, each distinct root contributes at least 1 to the drop in \(V\).

\noindent \textbf{Step 3: Effect of non-root points on \(V\).} \\
Between roots of \(p\), the sequence \(\mathcal{S}_p(t)\) can change only when some derivative other than \(p\) vanishes. Suppose at some \(c \in (a,b)\), \(p(c) \neq 0\) but \(p^{(k)}(c) = 0\) for some \(k \ge 1\). One can check that the number of sign variations in \(\mathcal{S}_p(t)\) can change by an even number at such points (by a local analytic argument). Thus, such events do not change the parity of \(V\) and the total drop in \(V\) from \(a\) to \(b\) equals the number of roots (counting multiplicity) plus an even nonnegative number.

\noindent \textbf{Step 4: Conclusion.} \\
Let \(a = t_0 < t_1 < \dots < t_k = b\) be points including all roots of \(p\) in \((a,b)\) and all points where any derivative vanishes (finitely many). On each \((t_{i-1}, t_i)\), \(V(\mathcal{S}_p(t))\) is constant. At each \(t_i\) that is a root of \(p\) of multiplicity \(m_i\), \(V\) drops by \(m_i + 2e_i\) for some integer \(e_i \ge 0\). At other \(t_i\), \(V\) drops by an even number \(\ge 0\).

Thus,
\[
V(a) - V(b) = \sum_{\text{roots } r_i} m_i + 2E
\]
for some integer \(E \ge 0\). Since \(N_{(a,b)}(p) \le \sum_{\text{roots } r_i} m_i\) (actually \(N\) = number of distinct roots, so \(N \le \sum m_i\)), we get
\[
N_{(a,b)}(p) \le V(a) - V(b),
\]
and \(V(a) - V(b) - N_{(a,b)}(p)\) is even and nonnegative.
\end{proof}

Corollary~\ref{cor:descartes} below is given to bound the number of \emph{distinct} positive real roots by no-zero coefficient terms. 

\begin{corollary}[Descartes' rule of signs]
\label{cor:descartes}
Let
\[
p(x) = a_0 + a_1x + \dots + a_nx^n \in \mathbb{R}[x],
\]
and suppose exactly \( t \) of the coefficients \( a_j \) are nonzero. Let \( V_{\mathrm{coeff}}(p) \) be the number of sign changes in the ordered coefficient sequence \( (a_0, a_1, \dots, a_n) \) after removing zeros. Then:

\begin{enumerate}
    \item The number of positive real roots of \( p \), counted with multiplicity, is either equal to \( V_{\mathrm{coeff}}(p) \) or less than it by an even number.
    \item In particular, the number of distinct positive real roots satisfies
    \[
    \#\{x>0 : p(x)=0\} \le V_{\mathrm{coeff}}(p) \le t-1.
    \]
\end{enumerate}
\end{corollary}

\begin{proof}
We prove both parts.

\noindent \textbf{Proof of (1):} \\
Consider the action of \( x = e^t \) substitution. The number of positive roots of \( p(x) \) equals the number of real roots of \( p(e^t) \). Apply Budan--Fourier theorem on the interval \( (0, \infty) \) in the \( x \)-variable, or equivalently, consider the limit of \( V(\mathcal{S}_p(0^+)) - V(\mathcal{S}_p(+\infty)) \).

At \( x \to 0^+ \), the Taylor expansion shows:
\[
p^{(k)}(0^+) = k! \cdot a_k.
\]
Thus the sequence \( \mathcal{S}_p(0^+) \) has signs matching \( (a_0, a_1, \dots, a_n) \) (up to positive factor \( k! \)), so
\[
V(\mathcal{S}_p(0^+)) = V_{\mathrm{coeff}}(p).
\]

As \( x \to +\infty \), \( p(x) \sim a_n x^n \) and \( p^{(k)}(x) \sim \frac{n!}{(n-k)!} a_n x^{n-k} \). For large \( x \), all derivatives have the same sign as \( a_n \), so
\[
V(\mathcal{S}_p(+\infty)) = 0.
\]

By Budan--Fourier theorem, the number of positive roots (counted with multiplicity) is
\[
N_{(0,\infty)}(p)\le V(\mathcal{S}_p(0^+)) - V(\mathcal{S}_p(+\infty)) = V_{\mathrm{coeff}}(p),
\]
and the difference is an even nonnegative integer.

\noindent \textbf{Proof of (2):} \\
The first inequality \( \#\{x>0 : p(x)=0\} \le V_{\mathrm{coeff}}(p) \) follows immediately from part (1), since the number of distinct positive roots is at most the number of roots counted with multiplicity.

For the second inequality \( V_{\mathrm{coeff}}(p) \le t-1 \): \\
If we have \( t \) nonzero coefficients, the maximum number of sign changes occurs if their signs alternate among each consecutive pair, giving at most \( t-1 \) sign changes between consecutive nonzero coefficients. Removing zeros doesn't increase the number of sign changes beyond this alternating maximum.
\end{proof}

Theorem~\ref{thm:khovanskii-univariate} below is given to provide bounds for distinct roots. 
\begin{theorem}[Univariate Fewnomial Complex Roots]\label{thm:khovanskii-univariate}
Let \( p(x) \) be a univariate polynomial with \( t \) nonzero terms. The number of distinct complex roots can be as large as the degree of the polynomial $p(x)$, but the number of distinct positive real roots is at most \( t - 1 \).
\end{theorem}

\begin{proof}
The complex root bound follows from the Fundamental Theorem of Algebra (at most degree \( d \) roots). The real root bound is Descartes' Rule of Signs given by Corollary~\ref{cor:descartes}, i.e., a polynomial with \( t \) terms can have at most \( t-1 \) sign changes, hence at most \( t-1 \) positive real roots.
\end{proof}

\begin{remark}[Multiplicity and parity]
The standard Descartes' Rule of Signs establishes an upper bound for the number of \emph{distinct} positive real roots. A more precise formulation states that the difference between the number of sign changes \( V_{\mathrm{coeff}}(p) \) and the number of positive real roots (\emph{counted with multiplicity}) is always an \emph{even nonnegative integer}.

This parity property follows directly from the Budan–Fourier theorem. It can also be understood via a continuity argument: when the coefficients of a polynomial are slightly perturbed to remove multiple roots, the quantity $V_{\mathrm{coeff}}(p) - N_{\mathrm{positive}}$ 
retains its parity. Indeed, each multiple root of multiplicity (m) splits into (m) distinct simple roots under such a perturbation, and the change in sign variations corresponds to the same parity as (m). Thus, the overall parity of the difference remains invariant.
\end{remark}

\begin{remark}[Fewnomials and Khovanskii]
Corollary~\ref{cor:descartes} provides an elementary and \emph{sharp} upper bound for univariate polynomials that depends only on the number of nonzero terms \( t \):
\[
\#\{ x > 0 : p(x) = 0 \} \le t - 1.
\]

This result is a special case of a much broader theory known as \emph{Khovanskii's fewnomial theory}. The latter gives a deep generalization to \emph{systems} of real polynomial equations: it bounds the number of isolated positive real solutions of a system in terms of the total number of \emph{distinct monomials} appearing across all equations. In general, this bound grows roughly exponentially in the number of monomials.

However, when Khovanskii's general bound is specialized to the univariate case (a single polynomial equation), it yields an upper bound that is exponential in \( t \). We note that the elementary bound from Descartes' Rule—namely \( t - 1 \)—is much sharper for univariate polynomials. This illustrates the special strength of Descartes' Rule in the one-dimensional setting.
\end{remark}

Following Example~\ref{exp:thm:khovanskii-univariate} is presented to illustrate Theorem~\ref{thm:khovanskii-univariate}.
\begin{example}\label{exp:thm:khovanskii-univariate}

Let \( f(x) = x^n - 1 \) (two terms).  
\begin{itemize}
  \item For \( \acute{m}=2 \):  
  \( f'(x) = n x^{n-1} \) has the unique root \( x=0 \).  
  Since \( f(0) = -1 \neq 0 \), we obtain exactly one root: \( x=0 \).  

  \item For \( \acute{m} \geq 3 \):  
  Higher derivatives vanish at \( x=0 \) as well, and still \( f(0)\neq 0 \).  
  Hence the unique root remains \( x=0 \).
\end{itemize}

\end{example}

Step 4 in Algorithm~\ref{alg:distinct_roots One Distinct Eigenvalue_D} can be obtained by Theorem~\ref{thm:khovanskii-univariate}. If we set the number of distinct roots of the polynomial $h$ in as $n_d$, the total number of distinct Jordan structures is ${n_d \choose 1}p(m)$, where $p(m)$ is the partition number and $m$ is the matrix $\bm{X}$ dimension.

\subsubsection{Dense Coefficients}

For dense polynomials, there exist several well-established root-counting algorithms for distinct roots.

\paragraph{Euclidean Algorithm for gcd.} 
Let \( h(x) \in \mathbb{C}[x] \) be a nonzero polynomial. To remove multiple roots and obtain a polynomial with the same roots but all simple, one computes
\[
q(x) = \gcd\big(h(x), h'(x)\big).
\]
Then \( h(x)/q(x) \) is the \emph{square-free part} of \( h \), which has the same set of distinct roots as \( h \), but each with multiplicity 1.

More precisely:
\begin{itemize}
    \item If \( h(x) = c \prod_{i=1}^k (x - r_i)^{m_i} \) with \( m_i \geq 1 \) and \( r_i \) distinct, then
    \[
    q(x) = \prod_{i=1}^k (x - r_i)^{m_i - 1}.
    \]
    \item Thus,
    \[
    \frac{h(x)}{q(x)} = c \prod_{i=1}^k (x - r_i),
    \]
    which has exactly the same roots \( r_1, \dots, r_k \), all simple.
\end{itemize}

Therefore, the factorization of \( h(x)/q(x) \) indeed yields the distinct roots of \( h \). Note that the polynomial $h$ is obtained by Step 3 in Algorithm~\ref{alg:distinct_roots One Distinct Eigenvalue_D}.

\paragraph{Sturm Sequence Algorithm.} 
This method is based on polynomial remainder sequences and allows for the exact counting of distinct real roots of \( f \) within any given interval.  

\paragraph{Isolation and Root Refinement.} 
One may use the Bernstein basis or Descartes’ method to isolate root intervals. These intervals can then be refined using Newton iteration or bisection. The overall complexity of this procedure is polynomial in the degree \( n \) and in the size of the coefficients.  

\paragraph{Numerical Polynomial Root Finding.} 
Classical numerical algorithms such as the Jenkins–Traub method, the Aberth method, or the Durand–Kerner method can approximate all roots simultaneously. Distinctness of the roots is then determined by clustering the approximated values~\cite{Stoer2002Introduction}.  

The counting of the number different Jordan structures are same in previous section about sparse coefficients. 

\subsection{Multiple Distinct Eigenvalues}\label{sec':Multiple Distinct Eigenvalues_D}

From Theorem~\ref{thm: Spectral Mapping Theorem for Single Variable}, we have 
\begin{eqnarray}
f(\bm{X})&=& \bm{U}\left(\bigoplus\limits_{k=1}^{K}\bigoplus\limits_{i=1}^{\alpha_{k}^{(\mathrm{G})}}f(\bm{J}_{m_{k,i}}(\lambda_{k}))\right)\bm{U}^{-1}.
\end{eqnarray}
The purpose of this section is to examine the existence of a matrix $\bm{X}$ such that, for a given polynomial $f$, the matrix $f(\bm{X})$ has all off-diagonal entries equal to zero. Furthermore, if such matrices $\bm{X}$ exist, we investigate how many solutions arise in both the sparse coefficients of $f$ and dense coefficients of  $f$.

\subsubsection{Sparse Coefficients}

Suppose the polynomial $h$ is obtained based on Theorem~\ref{thm:S Counting distinct roots}, we have following Corollary~\ref{cor:Practical counting via squarefree part} about the number of distinct roots $x$ satisfying  \(f(x)\neq0\) and \(f',\dots,f^{(\acute{m}-1)}\).

\begin{corollary}[Practical counting via squarefree part]\label{cor:Practical counting via squarefree part}
Let \(h_{\mathrm{sf}}(x)\) denote the squarefree part of \(h(x)\),
\[
h_{\mathrm{sf}}(x)=\frac{h(x)}{\gcd(h(x),h'(x))}.
\]
Then the number of distinct points \(x\) with \(f(x)\neq0\) and \(f',\dots,f^{(\acute{m}-1)}\) vanishing equals \(\deg h_{\mathrm{sf}}\).
\end{corollary}

\begin{proof}
The squarefree part \(h_{\mathrm{sf}}\) has the same distinct roots as \(h\) but with multiplicity \(1\). Hence counting distinct roots of \(h\) is equivalent to computing \(\deg h_{\mathrm{sf}}\).
\end{proof}

Following Theorem~\ref{thm:Sufficient and checkable conditions for at least K distinct points} is given to provide conditions for at least $K$ distinct roots satisfying below:  $f(x)\neq 0$ and $f'(x)=f''(x)=\cdots=f^{(\acute{m}-1)}(x)=0$. 

\begin{theorem}[Sufficient and checkable conditions for at least \(K\) distinct points]\label{thm:Sufficient and checkable conditions for at least K distinct points}
With notation as above, the following are sufficient (and easily checkable) conditions that guarantee \(\#\mathcal{Z}\ge K\), where $\mathcal{Z}$ is the number of distinct roots for the polynomial $f$:
\begin{enumerate}
  \item \(\deg h_{\mathrm{sf}}\ge K\).
  \item Equivalently, \(\deg h \ge K\) and \(\gcd(h,h')\) has degree at most \(\deg h-K\) (so that after removing repeated factors at least \(K\) distinct linear factors remain).
  \item Computationally: compute \(g=\gcd(f',\dots,f^{(\acute{m}-1)})\), \(d=\gcd(f,g)\), form \(h=g/d\), then compute the squarefree part \(h_{\mathrm{sf}}=h/\gcd(h,h')\). If \(\deg h_{\mathrm{sf}}\ge K\) then there are at least \(K\) distinct points in \(\mathcal{Z}\).
\end{enumerate}
\end{theorem}

\begin{proof}

We prove the three items in order, reducing each to Corollary~\ref{cor:Practical counting via squarefree part}.

\medskip\noindent\textbf{(1)} By construction (and the discussion preceding the theorem) the polynomial \(h\) has exactly the same roots (counted with multiplicity) as the set of points \(z\) for which
\[
f'(x)=f''(x)=\cdots=f^{(\acute{m}-1)}(x)=0
\]
but without the roots where additionally \(f(x)=0\) (those have been removed when forming \(h\); see item (3) for the explicit construction). Therefore the number of \emph{distinct} points \(z\) with \(f(x)\neq0\) and \(f',\dots,f^{(\acute{m}-1)}\) vanishing is exactly the number of distinct roots of \(h\). By Corollary~\ref{cor:Practical counting via squarefree part} this number equals \(\deg h_{\mathrm{sf}}\). Hence if \(\deg h_{\mathrm{sf}}\ge K\) then \(\#\mathcal Z\ge K\), proving item 1.

\medskip\noindent\textbf{(2)} The equality
\[
h_{\mathrm{sf}}=\frac{h}{\gcd(h,h')}
\]
implies
\[
\deg h_{\mathrm{sf}}=\deg h-\deg\big(\gcd(h,h')\big).
\]
Thus \(\deg h_{\mathrm{sf}}\ge K\) is equivalent to
\[
\deg h-\deg\big(\gcd(h,h')\big)\ge K,
\]
i.e.
\[
\deg\big(\gcd(h,h')\big)\le \deg h-K.
\]
Interpreting \(\gcd(h,h')\) as the polynomial collecting all repeated prime factors of \(h\), this condition means that after removing repeated factors from \(h\) at least \(K\) distinct linear factors remain. This proves the equivalence asserted in item 2.

\medskip\noindent\textbf{(3)} We now justify the computational recipe. Let
\[
g:=\gcd\big(f',f'',\dots,f^{(\acute{m}-1)}\big).
\]
By definition every root of \(g\) is a point where all derivatives \(f',\dots,f^{(\acute{m}-1)}\) vanish; conversely every point where those derivatives vanish is a root of \(g\). Some of those roots might also satisfy \(f(x)=0\); to exclude those we form
\[
d:=\gcd(f,g),
\]
so that \(d\) captures exactly the common factors corresponding to points where \(f=0\) and the derivatives vanish simultaneously. Then
\[
h:=\frac{g}{d}
\]
is a polynomial whose roots are precisely the points where the derivatives vanish but \(f\neq0\), counted with their multiplicities inherited from \(g\). Finally compute the squarefree part
\[
h_{\mathrm{sf}}=\frac{h}{\gcd(h,h')},
\]
which has the same distinct roots as \(h\) but with multiplicity \(1\). By Corollary~\ref{cor:Practical counting via squarefree part} the number of distinct such points equals \(\deg h_{\mathrm{sf}}\). Hence \(\deg h_{\mathrm{sf}}\ge K\) implies \(\#\mathcal Z\ge K\), which is exactly the computational criterion stated in item 3. This completes the proof of Theorem~\ref{thm:Sufficient and checkable conditions for at least K distinct points}.
\end{proof}

\begin{remark}[Interpretation and useful checks]
\begin{enumerate}
  \item The polynomial \(g\) detects locations where the first \(\acute{m}-1\) derivatives simultaneously vanish. Dividing by \(d=\gcd(f,g)\) removes those locations where \(f\) also vanishes (we only want points with \(f\neq0\)).
  \item The squarefree test is important: a high degree \(h\) could nonetheless have few \emph{distinct} roots if it contains high multiplicity factors. Passing to the squarefree part gives the exact distinct-root count.
  \item Algorithmically the steps are standard polynomial operations (Euclidean gcd, differentiation, polynomial division) and are efficiently implemented in computer algebra systems.
  \item If one needs lower bounds without explicit computation of \(h\), one can derive combinatorial sufficient conditions from the structure of \(f\) (for example, if \(f'\) factors with several known coprime linear factors that are not zeros of \(f\), each such factor contributes at least one point to \(\mathcal{Z}\)).
\end{enumerate}
\end{remark}

\paragraph{Counting the number of Jordan Blocks}

We begin with the following settings.
\begin{itemize}
  \item Let \(n_d\) denote the total number of available distinct eigenvalues which can be obtained from Theorem~\ref{thm:Sufficient and checkable conditions for at least K distinct points}.
  \item We select \(K\) distinct eigenvalues from these (\(1 \leq K \leq n_d\)).
  \item The total matrix size is \(m\).
  \item For the chosen \(K\) eigenvalues, let
  \[
  \alpha_{1}^{(\mathrm{A})}, \alpha_{2}^{(\mathrm{A})}, \dots, \alpha_{K}^{(\mathrm{A})},
  \qquad 
  \sum_{k=1}^K \alpha_{k}^{(\mathrm{A})} = m,
  \]
  where \(\alpha_{k}^{(\mathrm{A})}\) is the algebraic multiplicity of eigenvalue \(\lambda_k\).
  \item Let \(p(n)\) denote the partition number of \(n\), i.e. the number of ways of decomposing \(n\) into a sum of positive integers, ignoring order.
\end{itemize}

\subparagraph{Jordan Structures for a Single Eigenvalue.}
For a fixed eigenvalue \(\lambda_k\) with algebraic multiplicity \(\alpha_{k}^{(\mathrm{A})}\), the possible Jordan structures correspond bijectively to the partitions of \(\alpha_{k}^{(\mathrm{A})}\).  
Hence the number of distinct Jordan block structures for \(\lambda_k\) is
\[
p\big(\alpha_{k}^{(\mathrm{A})}\big).
\]

\subparagraph{Fixed Multiplicity Vector.}
Given a multiplicity vector \(\big(\alpha_{1}^{(\mathrm{A})},\dots,\alpha_{K}^{(\mathrm{A})}\big)\), the number of distinct Jordan structures (up to block permutation) is
\[
\prod_{k=1}^K p\!\left(\alpha_{k}^{(\mathrm{A})}\right).
\]

\subparagraph{Summing over Multiplicity Vectors.}
For a fixed choice of \(K\) eigenvalues, the multiplicities vary over all positive integer vectors with sum \(m\).  
Therefore, the number of Jordan structures for this fixed eigenvalue set is
\[
\sum_{\substack{\alpha_{1}^{(\mathrm{A})}+\cdots+\alpha_{K}^{(\mathrm{A})}=m \\ \alpha_{k}^{(\mathrm{A})}\geq 1}}
\;\prod_{k=1}^K p\!\left(\alpha_{k}^{(\mathrm{A})}\right).
\]

\subparagraph{Choosing the Eigenvalues.}
The number of ways to choose the \(K\) distinct eigenvalues from the \(n_d\) available ones is
\(\binom{n_d}{K}\).

\subparagraph{Final Count.}
Hence the total number of distinct Jordan structures (treating block permutations as identical) is
\begin{equation}\label{eq:Jordan-structure-count}
\boxed{\;
\mathcal{N}(n_d,K,m)
=
\binom{n_d}{K}\,
\sum_{\substack{\alpha_{1}^{(\mathrm{A})}+\cdots+\alpha_{K}^{(\mathrm{A})}=m \\ \alpha_{k}^{(\mathrm{A})}\geq 1}}
\;\prod_{k=1}^K p\!\left(\alpha_{k}^{(\mathrm{A})}\right).
\;}
\end{equation}

\subsubsection{Dense Coefficients}

Although the approaches for the existence of $K$ distinct roots and the counting of differernt Jordan blocks for valid matrix $\bm{X}$ are same for sparse coeffeicients, we will discuss some computation variation for gcd in terms of the sparse $f$ and the dense $f$. 

In polynomial gcd computation, the structure of a polynomial—whether sparse or dense—significantly affects the choice of algorithm and computational efficiency. Let
\[
\gcd(f(x), g(x)) \quad \text{or more generally} \quad \gcd(f, f', f'', \dots, f^{(\acute{m}-1)})
\]
be the target gcd computation, where $g(x)=\gcd(f', f'', \dots, f^{(\acute{m}-1)})$.

A \textbf{sparse polynomial} is characterized by having relatively few non-zero terms compared to its degree. For example, \(f(x) = x^{1000} + x^3 + 1\) is highly sparse. Computation and storage focus on the non-zero terms rather than the full degree range. Algorithms for sparse polynomials typically include exponent-based gcd methods, which compute the gcd by analyzing the positions and exponents of non-zero terms, and randomized evaluation methods such as Zippel's algorithm, which reduce computation by evaluating the polynomial at random points~\cite{zippel1981newton}. FFT-based methods are generally inefficient for sparse polynomials because they introduce many unnecessary operations on zero coefficients. The complexity of sparse gcd algorithms is primarily determined by the number of non-zero terms, denoted \(t_f, t_g\), rather than the highest degree, making them especially efficient for high-degree polynomials with few terms.

In contrast, a \textbf{dense polynomial} has most of its coefficients non-zero. An example is \(f(x) = x^{100} + 2x^{99} + \dots + 1\). The evaluation of gcd of dense polynomials can be determined by classical Euclidean algorithms, FFT-based multiplication/division methods, and modular techniques that reduce coefficients modulo large primes or polynomials. These approaches utilize the continuity and regularity of coefficients. The computational complexity for dense polynomials depends on the polynomial degree  $\deg(f)$, rather than the number of terms, and is generally higher than sparse methods for extremely high-degree polynomials with few non-zero coefficients.

The main differences between sparse and dense gcd computation strategies can be summarized in Table~\ref{tab:gcd_sparse_dense}:

\begin{table}[h!]
\centering
\footnotesize
\setlength{\tabcolsep}{3pt}
\renewcommand{\arraystretch}{1.2}
\begin{tabularx}{\columnwidth}{|c|X|c|X|}
\hline
\textbf{Polynomial Type} & \textbf{Core Strategies} & \textbf{Complexity Measure} & \textbf{Notes} \\
\hline
Sparse & Exponent-based, randomized (e.g., Zippel) & Number of non-zero terms \(t_f, t_g\) & Efficient for high-degree, few-term polynomials. \\
\hline
Dense & Euclidean algorithm, FFT, modular techniques & Polynomial degree \(\deg(f)\) & Efficient for polynomials with most coefficients non-zero. \\
\hline
\end{tabularx}
\caption{Comparison of gcd computation strategies for sparse vs. dense polynomials.}
\label{tab:gcd_sparse_dense}
\end{table}

In summary, sparse polynomial algorithms leverage the small number of non-zero terms to reduce computation, while dense polynomial algorithms exploit coefficient continuity using classical or FFT-based methods. The choice of algorithm depends critically on the polynomial's structure to ensure computational efficiency.

\bibliographystyle{IEEEtran}
\bibliography{ZeroMatrixCompare_CountRealRoots_Bib}

\end{document}